\documentclass[12pt]{article}
\def\cleverefoptions{capitalize,noabbrev}
\usepackage[amsmath,cleveref]{e-jc}

\usepackage{graphicx}
\usepackage{mathtools}
\usepackage{mathrsfs}
\usepackage{bm}
\usepackage{amsfonts}
\usepackage{physics}
\usepackage{cite}
\usepackage{enumitem}
\usepackage{booktabs}
\usepackage{threeparttable}
\usepackage{multirow}
\usepackage{makecell}

\newcommand\set[2]{ \left\{ {#1} \; \colon \; {#2} \right\} }
\newcommand\setnd[1]{ \left\{ #1 \right\} }
\newcommand\floor[1]{\left\lfloor{#1} \right\rfloor}
\newcommand\card[1]{\left|{#1}\right|} 

\newcommand{\inp}[2]{\left\langle #1, #2 \right\rangle}
\newcommand\ZZ{\mathbb{Z}} 
\newcommand\CC{\mathbb{C}} 
\newcommand{\F}{\mathbb{F}}
\newcommand{\defeq}{\triangleq}
\newcommand{\dual}[1]{#1^\bot}

\newcommand{\binomial}[2]{{#1 \choose #2}}
\DeclareMathOperator{\supp}{\mathrm{supp}}
\DeclareMathOperator{\wt}{\mathrm{wt}}
\DeclareMathOperator{\cay}{\mathrm{Cay}}

\crefname{theorem}{Theorem}{Theorems}
\crefname{lemma}{Lemma}{Lemmas}
\crefname{corollary}{Corollary}{Corollaries}
\renewcommand{\autoref}[1]{\cref{#1}}

\dateline{TBD}{TBD}{TBD}


\Copyright{The authors. Released under the CC BY-ND license (International 4.0).}

\title{Quantum Chromatic Number of  \\
Subgraphs of Orthogonality Graphs and \\ the Distance-$2$ Hamming Graph}

\author{Tao Luo\authornote{1,2,4}
\and
Yu Ning\authornote{3,4}
\and
Xiande Zhang\authornote{1,2,4,5}
}

\authortext{1}{
School of Mathematical Sciences, 
University of Science and Technology of China, Hefei 230026, Anhui, China  
(\email{luotao01@mail.ustc.edu.cn}, \email{drzhangx@ustc.edu.cn}).}
\authortext{2}{Hefei National Laboratory, University of Science and
Technology of China, Hefei 230088, China}
\authortext{3}{Hefei National Laboratory, Hefei 230088, China (\email{sirning@mail.ustc.edu.cn}).}
\authortext{4}{Authors are listed alphabetically.}
\authortext{5}{Corresponding author (\email{drzhangx@ustc.edu.cn}).}

\begin{document}
\maketitle


\begin{abstract}
The determination of the quantum chromatic number of graphs has attracted
considerable attention recently. However, there are few families of graphs whose
quantum chromatic numbers are determined. A notable exception is the family of
orthogonality graphs, whose quantum chromatic numbers are fully determined. In
this paper, we extend these results by determining the exact quantum chromatic
number of several subgraphs of the orthogonality graphs. Using the technique of
combinatorial designs, we also determine the quantum chromatic number of the
distance-$2$ Hamming graph, whose edges consist of binary vectors of Hamming
distance $2$, for infinitely many length.
\end{abstract}

\section{Introduction}
\label{sec:introduction}
A $c$-\emph{coloring} of a graph $G$ with vertex set $V$, is a map $V \to [c]
\defeq \setnd{1,2,\dots,c}$ such that adjacent vertices receive distinct colors
in $[c]$. The \emph{chromatic number} $\chi(G)$ of $G$ is the smallest $c$ such
that $G$ admits a $c$-coloring, which has been thoroughly studied in the
literature \cite[Chapter 5]{GraphTheory}. In the past decades, motivated by the
study in quantum information theory \cite{QuantumVSClassical,CostSimulating},
the $c$-coloring game of graphs was proposed, which can be though of as a
quantum generalization of the classical graph coloring if the players are
accessible to quantum resources
\cite{OnTheQuantumChromaticNumber,QuantumProtocol}.

In the $c$-coloring game of the graph $G$, two players, Alice and Bob, must
convince a referee that they possess a $c$-coloring of $G$. Both of the players
receive a vertex $v,w \in V$ chosen randomly by the referee, respectively, and
must then respond with a color $c_v,c_w \in [c]$ to the referee. During the
game, Alice and Bob are not allowed to communicate with each other, but they are
allowed to share classical or quantum resources and agree on some protocol a
priori, which may depend on $G$. The referee is convinced that Alice and Bob
possess a $c$-coloring of $G$, if $c_v = c_w$ when $v = w$ and $c_v \neq c_w$
when $v$ and $w$ are adjacent. While classical players can win the $c$-coloring
game of $G$ with probability $1$ if and only if a classical $c$-coloring of $G$
exists, for particular graphs, quantum players, who have access to a shared
entangled quantum state and may perform local measurements on the shared
entangled state, may still win with probability $1$ even if classical
$c$-coloring of $G$ does not exist. The smallest number $c$ for which a quantum winning
protocol exists is called the \emph{quantum chromatic number} of $G$, denoted as
$\chi_q(G)$ \cite{OnTheQuantumChromaticNumber,QuantumHomomorphisms}. Notably,
there exist families of graphs for which $\chi_q(G)$ is strictly smaller than
$\chi(G)$, demonstrating a quantum advantage
\cite{ExactHadamardGraphs,Feng,SmallGraphs}. The quantum chromatic number is
also related to other important parameters of graphs \cite{QuantumHomomorphisms}
or objects in quantum information theory \cite{KSSet}.

In general, it is NP hard to determine the quantum chromatic number of a graph
\cite{JiZhengfeng}. To our knowledge, there are only few classes of graphs whose
quantum chromatic number was determined. For example, the complete graph on $n$
vertices has quantum chromatic number $n$ \cite{OnTheQuantumChromaticNumber};
bipartite graphs have quantum chromatic number $2$
\cite{OnTheQuantumChromaticNumber}; $\chi(G) = 3$ if and only if $\chi_q(G) = 3$
\cite{OnTheQuantumChromaticNumber}, in particular, odd cycles have quantum
chromatic number $3$. Beyond these special graphs, the \emph{orthogonality
graph} $O_{n,p}$ is an important class of graphs studied in this topic, whose
vertices are vectors of length $n$ with entries the $p$-th roots of unity, and
two vectors are adjacent if and only if they are orthogonal in $\CC^{n}$. When
$p$ is a prime and $p \mid n$, $O_{n,p}$ is isomorphic to the Cayley graph over
$\ZZ_p^n$ with generating set consisting of all balanced vectors in $\ZZ_p^n$.
Especially when $p=2$, the graph $O_{4t,2}$  (also known as the Hadamard graph)
and one of its subgraph both have quantum chromatic number $4t$
\cite{ExactHadamardGraphs}.
For general $p$, Cao \emph{et, al.} \cite{SDU} recently
determined the quantum chromatic number of $O_{lp,p}$ when $l$  is
large enough with respect to $p$ and $l(p-1)$ is even.

Motivated by the above results, we study the quantum chromatic number of
subgraphs of  orthogonality graphs. In particular, we  determine the quantum
chromatic numbers of five classes of subgraphs of the orthogonality graphs
$O_{4t,2}$ and $O_{3l,3}$. Further, using the technique of combinatorial
designs, we also determine the quantum chromatic number of the distance-$2$
Hamming graph, whose edges correspond to pairs of binary vectors of Hamming
distance $2$, that is, the Cayley graph over $\ZZ_2^n$ generated by all vectors
of Hamming weight $2$ for infinitely many $n$. The results mentioned above in
the literature and in this paper are summarized in \autoref{tab:results} for
readers' convenience.

\begin{table}[t]
\centering
\begin{threeparttable}
\caption{Known Quantum Chromatic Numbers of Graphs}
\label{tab:results}
\begin{tabular}{l|c|c}
\toprule
Graphs $G$ & $\chi_q(G)$ & References \\
\midrule
$K_n$ (the complete graph on $n$ vertices)  & $n$ & \cite{OnTheQuantumChromaticNumber} \\
bipartite graphs & 2 & \cite{OnTheQuantumChromaticNumber} \\
graphs with $\chi(G) = 3$ & $3$ & \cite{OnTheQuantumChromaticNumber} \\
$O_{4,4}$  & $4$ & \cite{OnTheQuantumChromaticNumber} \\
$O_{4t,2}\cong \cay(\ZZ_2^{4t},(2t,2t))$ & $4t$ & \cite{ExactHadamardGraphs,Feng} \\
$\cay(\ZZ_2^{4t-1},(2t-1,2t))$ & $4t$ & \cite{Feng} \\
\hline
\makecell[l]{$\cay(\ZZ_p^{lp},(l,l,\dots,l))$, $p$ is an integer, \\ $l$ is
 large enough and $l(p-1)$ is even} & $lp$ 
 & \cite{SDU} \\
 \hline
$\cay(\F_q^{q^l},(q^{l-1},q^{l-1},\dots,q^{l-1}))$, $q$ is a prime power& $q^l$ & \cite{SDU} \\
\midrule
$O_{3l,3} \cong \cay(\ZZ_3^{3l},(l,l,l))$ & $3l$ & \autoref{cor:3l} \\
$\cay(\ZZ_2^{4t-1},(2t,2t-1) \cup (2t-1,2t))$ & $4t$ &
\autoref{corollary:easy} \\
$\cay(\ZZ_2^{4t-2},(2t-1,2t-1) \cup (2t-2,2t))$ & $4t$ &
\autoref{corollary:easy} \\
$\cay(\ZZ_3^{3l-1},(l-1,l,l)\cup (l,l-1,l) \cup (l,l,l-1))$  & $3l$ &
\autoref{corollary:easy}  \\
$\cay(\ZZ_p^{lp-1},(l-1,\dots,l) \cup \dots \cup (l,\dots,l-1))$ & $lp$ & 
\autoref{corollary:easy}\\
$\cay(\ZZ_3^{3l-1},(l-1,l,l))$ & $3l$ & \autoref{theorem:subgraph_eigenvalue_3} \\
$\cay(\ZZ_3^{3l-2},(l-2,l,l)\cup (l-1,l-1,l) \cup (l-1,l,l-1))$& $3l$ &
\autoref{theorem:subgraph_3_III} \\
\midrule
$\cay(\ZZ_2^{q},(q-2,2))$, $q \equiv 3 \pmod{4}$ is a prime power & $q+1$ &
\autoref{tab:BIBD} \\
\hline
$\cay(\ZZ_2^{2^{t+2}-1},\qty(2^{t+2}-3,2))$ & $2^{t+2}$ & \autoref{tab:BIBD} \\
\hline
\makecell[l]{$\cay(\ZZ_2^{q^2+2q},(q^2+2q-2,2))$, \\ $q$ and $q+2$ are both odd prime powers}
& $q^2+2q+1$ & \autoref{tab:BIBD} \\
\bottomrule
\end{tabular}
\begin{tablenotes}
\item[1] In the table, $\cay(\ZZ_2^{4t},(2t,2t))$ is the Cayley graph over
$\ZZ_2^{4t}$ with a generating set consisting of vectors of type $(2t,2t)$ (see
\autoref{subsec:orthogonality_graphs}); Similar for other Cayley graphs.
\end{tablenotes}
\end{threeparttable}
\end{table}

The paper is organized as follows. \autoref{sec:preliminary} introduces
necessary preliminaries on quantum chromatic numbers and orthogonality graphs.
In \autoref{sec:3}, we study the spectrum of the orthogonality graph $O_{3l,3}$,
which will be used to determine the quantum chromatic number of its subgraphs.
In \autoref{sec:subgraphs}, we determine the quantum chromatic number of five
classes of subgraphs of $O_{4t,2}$ and $O_{3l,3}$. In \autoref{sec:Hn2}, with
the technique of combinatorial designs, we determine the quantum chromatic
number of the distance-$2$ Hamming graph, that is, the Cayley graph over
$\ZZ_2^n$ generated by vectors of Hamming weight $2$ for certain $n$. Finally,
we conclude this paper in \autoref{sec:conclusion} and discuss open problems.

\section{Preliminary}
\label{sec:preliminary}
In this section, we introduce the notions and related results about the
quantum chromatic number of orthogonality
graphs.

Let $G = (V,E)$ be a simple graph. 
The quantum chromatic number of $G$ was motived by
quantum information theory and was originally formulated in the framework of
non-local games as described in Introduction (also see
\cite{OnTheQuantumChromaticNumber} and the references therein). We adopt
the mathematical definition of quantum chromatic numbers as in
\cite{KSSet,SpectralLB,QuantumHomomorphisms}. A \emph{quantum $c$-coloring} of
$G$ is a collection of orthogonal projections $P_{v,i} : \CC^d \to \CC^d$ for $v
\in V$ and $i \in [c]$ satisfying the following two conditions:
\begin{itemize}
\item[(i)] for all $v \in V$, $\sum_{i \in [c]} P_{v,i} = I_d,$ and
\item[(ii)] for all adjacent vertices $v \sim w$ and all $i \in [c]$, $P_{v,i}P_{w,i} = 0.$
\end{itemize}
The \emph{quantum chromatic number} $\chi_q(G)$ of $G$ is the smallest $c$ such that
$G$ admits a quantum $c$-coloring for some dimension $d \ge 1$.

A homomorphism from a graph $G=(V,E)$ to a graph $H=(V',E')$ is a map $f:
V\rightarrow V'$ such that if $\{u,v\}\in E$, then $\{f(u),f(v)\}\in E'$. It is
known that \cite{OnTheQuantumChromaticNumber} if there is a homomorphism of
graphs $G \to H$, then
\begin{equation}\label{eq:subg}
  \chi_q(G) \le \chi_q(H).
\end{equation}In particular, if $G$ is a subgraph of  $H$, then
$\chi_q(G) \le \chi_q(H)$.

\subsection{Orthogonality Graphs}
\label{subsec:orthogonality_graphs}
Let $p \ge 2$ be an integer, $\zeta_p \defeq e^{2 \pi \sqrt{-1}/p}$ be the
$p$-th root of unity, and $U_p \defeq \set{\zeta_p^i}{i \in \ZZ_p}$ be the group
generated by $\zeta_p$, where $\ZZ_p$ denotes the ring of integers modulo $p$.
The \emph{orthogonality graph} $O_{n,p}$ has vertex set $U_p^n$, and two
vertices $\mathbf{v},\mathbf{w} \in U_p^n$ are adjacent if and only if
$\mathbf{v}$ and $\mathbf{w}$ are orthogonal with respect to the standard
Hermitian inner product on $\CC^{n}$. 

Let $G$ be a graph with vertex set $V$. An \emph{orthogonal representation} of
$G$ is a map $\phi : V \to \CC^d$ for some $d \ge 1$, such that $\phi(v)$ and
$\phi(w)$ are orthogonal whenever $v$ and $w$ are adjacent in $G$.
The orthogonal representation $\phi$ is \emph{flat} if for all $v \in
V$, each component of $\phi(v)$ is of the same modulus. It is known that if $G$
admits a flat orthogonal representation $\phi : V \to \CC^d$, then one can
construct a quantum $d$-coloring of $G$ from this representation
\cite{Oddities}, and thus $\chi_q(G) \le d$.

The orthogonality graph $O_{n,p}$ is naturally equipped with a flat orthogonal
representation over $\CC^n$. Thus,
\begin{equation}
	\label{eqn:ub}
\chi_q(O_{n,p}) \le n.
\end{equation}	
For the lower bound of $\chi_q(O_{n,p})$, there exists a  spectral lower bound from \cite{SpectralLB},
\begin{equation}
	\label{eqn:spectral_lb}
\chi_q(G) \ge 1 - \frac{\lambda_{\max}(G)}{\lambda_{\min}(G)},
\end{equation} where
$\lambda_{\max}(G)$ and $\lambda_{\min}(G)$ are the largest and smallest
eigenvalues of  $G$, respectively.

To compute $\lambda_{\max}$ and $\lambda_{\min}$ of $O_{n,p}$,
we consider it as Cayley graphs over $\ZZ_p^n$, whose eigenvalues
can be expressed conveniently in terms of characters of abelian groups.
For $\mathbf{v} = (v_1,\dots,v_n) \in \ZZ_p^n$,
the \emph{type} of $\mathbf{v}$ is the sequence $T(\mathbf{v}) \defeq
(t_0,t_1,\dots,t_{p-1})$ given by
\begin{equation}
	\label{eqn:type}
\wt_i(\mathbf{v}) \defeq t_i \defeq \card{\set{j \in [n]}{v_j = i}}, \quad i \in \ZZ_p.
\end{equation}
It is clear that $T(\mathbf{v})$ is an ordered non-negative $p$-partition of
$n$. For any such partition $T = (t_0,\dots,t_{p-1})$ of $n$, we denote by
$\ZZ_p^n(t_0,\dots,t_{p-1})$ the set of vectors in $\ZZ_p^n$ of type $T$. When
$p\mid n$, vectors of type $(n/p,n/p,\ldots,n/p)$ are called {\it balanced}
vectors. Under these notations, when $p$ is a prime and $p\mid n$,
\begin{equation} O_{n,p} \cong \cay(\ZZ_p^n,\ZZ_p^n(n/p,n/p,\ldots,n/p)).
\end{equation}
Note that when $p$ is prime and $p \nmid n$, $O_{n,p}$ is empty.

When $p = 2$, if $n \equiv 2 \pmod{4}$, then $O_{n,2}$ is bipartite, and thus
$\chi_q(O_{n,2}) = \chi(O_{n,2}) = 2$ \cite{OnTheQuantumChromaticNumber}. If $n
\equiv 0 \pmod{4}$, then $\chi_q(O_{n,2}) = n$
\cite{OnTheQuantumChromaticNumber,Feng} while $\chi(O_{n,2})$ is exponential in
$n$, showing the advantage of quantum entanglement in the coloring game of
$O_{n,2}$.

When $p=3$, if $3 \mid n$,
the quantum chromatic number of $O_{n,3}$ was shown to be $n$ in \cite{SDU} for
large $n$. In fact, the authors of  \cite{SDU} showed
$\chi_q(\cay(\ZZ_p^n,\ZZ_p^n(n/p,\dots,n/p)))=n$ for every integer $p$ and
large $n$ satisfying $p \mid n$ and $(p-1)n/p$ even. Since
$\cay(\ZZ_p^n,\ZZ_p^n(n/p,\dots,n/p))$ is a subgraph of $O_{n,p}$ for
non-prime $p$, by \autoref{eq:subg}, we know that $\chi_q(O_{n,p}) = n$ for
large $n$ satisfying $p \mid n$ and $(p-1)n/p$ even.

In this paper, we focus on determining the quantum chromatic number of  several
classes of subgraphs of $O_{n,p}$ with $p=2,3$. The upper bound is trivially
$n$, while the main work is to give the matching spectral lower bound by
determining the smallest eigenvalue of $O_{n,p}$, which is the one with the
second largest absolute value. 

\subsection{Eigenvalues of Cayley Graphs}
\label{subsec:eigenvalue_Hamming}
We recall the eigenvalues of a Cayley graph. Suppose  
$G = \cay(A,S)$ is a Cayley graph over a finite abelian group $A$ generated by
an inversion-closed subset $S \subset A \setminus \setnd{0}$. For $a \in A$, we
denote by $\chi_a : A \to \CC^*$ the character of $A$ corresponding to $a$
\cite[Chapter 5]{FiniteFields}. Then all eigenvalues of $G = \cay(A,S)$ are
given by (see \cite[Chapter 1]{SpectraCayleyGraphs})
\begin{equation}
	\label{eqn:eigenvalue_Cayley_graph}
	\lambda(a) \defeq \sum_{s \in S} \chi_a(s), ~a\in A.
\end{equation}
For Cayley graphs over $\ZZ_p^n$, it is known that all the characters are
\cite[Chapter 5]{FiniteFields}
\begin{equation}
\label{eqn:character_p}
\chi_{\mathbf{v}} : \ZZ_p^n \to \CC^*, \quad \mathbf{x} \mapsto
\zeta_p^{\mathbf{v} \cdot \mathbf{x}},
\end{equation}
where $\mathbf{v} = (v_1,\dots,v_n), \mathbf{x} = (x_1,\dots,x_n) \in \ZZ_p^n$
and $\mathbf{v} \cdot \mathbf{x} = \sum_{i=1}^n v_ix_i$.

\subsubsection{$p=2$}\label{subsecp2} 
For simplicity, denote $L_r \defeq
\ZZ_2^n(n-r,r)$ for $r \in [0,n]$, i.e., the  set of vectors in $\ZZ_2^n$ of
Hamming weight $r$. Consider $\cay(\ZZ_2^n,L_r)$. For $\mathbf{v} \in \ZZ_2^n$
with $\wt(\mathbf{v}) = w$, by \autoref{eqn:eigenvalue_Cayley_graph}, the
eigenvalue corresponding to $\mathbf{v}$ is given by
\begin{equation}
	\label{eqn:Krawtchouk}
\lambda(\mathbf{v}) = \sum_{\mathbf{b} \in L_r}(-1)^{\mathbf{v} \cdot \mathbf{b}}
= \sum_{j}(-1)^j\binom{w}{j}\binom{n-w}{r-j} =
\qty((x+y)^{n-w}(x-y)^{w})[n-r,r] \defeq K_r(w).
\end{equation}
In \autoref{eqn:Krawtchouk}, $g(x,y)[n-r,r]$ stands for the
coefficient of the monomial $x^{n-r}y^r$ in $g(x,y)$, and
$K_r(w)$ is known as the \emph{Krawtchouk polynomial} \cite{SmallestEigenvalue}.
So the eigenvalue $\lambda(\mathbf{v})$ depends only on the Hamming weight $w$
of $\mathbf{v}$, and $K_r(w)$, $w \in [0,n]$ are all  eigenvalues of
$\cay(\ZZ_2^n,L_r)$ without counting multiplicities.
The largest eigenvalue of $\cay(\ZZ_2^n,L_r)$ is given by
$$
\lambda(\mathbf{0}) = K_r(0) = \card{L_r} = \binom{n}{r} = (-1)^r K_r(n).
$$
In particular, $K_r(0)$ and $K_r(n)$ are eigenvalues of the largest absolute
value. Moreover, when $r$ is odd, $K_r(n)=-\binomial{n}{r}$ is the smallest
eigenvalue of $\cay(\ZZ_2^n,L_r)$. For even $r\geq \frac{n}{2}$, the smallest
eigenvalue of $\cay(\ZZ_2^n,L_r)$ has been determined in
\cite{SmallestEigenvalue}. When $n\equiv 0\pmod 4$ and $r=\frac{n}{2}$, the
spectral lower bound shows $\chi_q(O_{n,2}) \geq n$, and hence $\chi_q(O_{n,2})
= n$ \cite{ExactHadamardGraphs}.

\subsubsection{$p=3$}\label{subsec:coding_approach}
Let $n=3l$. For simplicity, let $B_3 \defeq \ZZ_3^n(l,l,l)$ be
the set of balanced vectors in $\ZZ_3^n$. Then $O_{n,3} \cong \cay(\ZZ_3^n,B_3)$.
For each $\mathbf{v} \in \ZZ_3^n$ and  $a \in \ZZ_3$, let
\begin{equation}
B_3(\mathbf{v},a) \defeq \set{\mathbf{b} \in B_3}{\mathbf{v} \cdot \mathbf{b} = a}.
\end{equation}
Clearly, $ \mathbf{b} \mapsto -\mathbf{b}$ is a bijection between $B_3(\mathbf{v},1)$ and $B_3(\mathbf{v},2)$.
Thus, $\card{B_3(\mathbf{v},1)} = \card{B_3(\mathbf{v},2)}$ and
$\card{B_3(\mathbf{v},0)} + 2\card{B_3(\mathbf{v},1)} = \card{B_3} =
\binomial{n}{l,l,l}$. Then by
\autoref{eqn:eigenvalue_Cayley_graph}, we have the eigenvalue
\begin{align}
\lambda({\mathbf{v}})
&= \sum_{\mathbf{b} \in B_3(\mathbf{v},0)} \zeta_3^{\mathbf{v} \cdot \mathbf{b}} +
\sum_{\mathbf{b} \in B_3(\mathbf{v},1)} \qty(\zeta_3^{\mathbf{v} \cdot \mathbf{b}} +
\zeta_3^{-\mathbf{v} \cdot \mathbf{b}}) \\
&= \card{B_3(\mathbf{v},0)} - \card{B_3(\mathbf{v},1)} \\
&= \frac{3}{2}\card{B_3(\mathbf{v},0)} - \frac{1}{2}\card{B_3}. \label{eqn:eigenvalue3_code}
\end{align}

Now, we introduce some notions from coding theory to express
\autoref{eqn:eigenvalue3_code}. An $[n,k,d]$ ternary code $C$ is a
$\ZZ_3$-linear subspace of $\ZZ_3^n$ of dimension $k$ and minimum Hamming weight
$d$. The dual code of $C$ is defined as
\begin{equation}
	\label{eqn:dual_code}
	\dual{C} \defeq \set{\mathbf{v} \in \ZZ_3^n}{\forall~ \mathbf{c} \in C,
	\mathbf{v} \cdot \mathbf{c} = 0}.
\end{equation}
The (complete) weight enumerator of $C$ is defined as
\begin{equation}
	\label{eqn:weight_enumerator}
	A_C(x,y,z) \defeq \sum_{\mathbf{c} \in C}
	x^{\wt_0(\mathbf{c})} y^{\wt_1(\mathbf{c})} z^{\wt_2(\mathbf{c})}
	\defeq \sum_{r+s+t = n}A_C[r,s,t]x^ry^sz^t,
\end{equation}
where $\wt_0,\wt_1,\wt_2$ are defined in \autoref{eqn:type}, and $A_C[r,s,t]$ is
the coefficient of the monomial $x^ry^sz^t$ in $A_C(x,y,z)$, which is also
 the number of codewords in $C$ of type $(r,s,t)$. The weight enumerators
of $C$ and $\dual{C}$
are related by the MacWilliams transform \cite{SelfDual} as follows
\begin{equation}
	\label{eqn:MacWilliams}
A_{\dual{C}}(x,y,z) = \frac{1}{\card{C}}A_C(x+y+z,x+\zeta_3 y + \zeta_3^2 z,
x+\zeta_3^2 y + \zeta_3 z).
\end{equation}
With these notions, we express \autoref{eqn:eigenvalue3_code} in
terms of coefficients of the weight enumerator of some code in the following.
For $\mathbf{v} \in \ZZ_3^n$, let
\begin{equation}
C_\mathbf{v} \defeq \begin{cases}
	\setnd{\mathbf{0}}, & \mathbf{v} = \mathbf{0}, \\
	\setnd{\mathbf{0},\mathbf{v},-\mathbf{v}}, & \mathbf{v} \neq \mathbf{0},
\end{cases}
\end{equation}
be the code generated by $\mathbf{v}$. Note that $C_{\mathbf{v}}$ is of
dimension $0$ when $\mathbf{v} = 0$, and dimension $1$ when $\mathbf{v} \neq
0$. Assume that $T(\mathbf{v}) = (t_0,t_1,t_2)$, then
\begin{equation}
A_{C_\mathbf{v}}(x,y,z) = \begin{cases}
x^n & \mathbf{v} = \mathbf{0} \\
x^n + x^{t_0}y^{t_1}z^{t_2} + x^{t_0}y^{t_2}z^{t_1} & \mathbf{v} \neq \mathbf{0}.
\end{cases}
\end{equation}
Clearly, $A_{C_\mathbf{v}}(x,y,z)$ and $A_{\dual{C_\mathbf{v}}}(x,y,z)$ depend only on
the type $T(\mathbf{v}) = (t_0,t_1,t_2)$ of $\mathbf{v}$, so we define
\begin{align}
A_{(t_0,t_1,t_2)}(x,y,z) &\defeq A_{C_\mathbf{v}}(x,y,z), \\
A_{\dual{(t_0,t_1,t_2)}}(x,y,z) &\defeq A_{\dual{C_\mathbf{v}}}(x,y,z).
\end{align}
By definition, $\card{B_3(\mathbf{v},0)} =
A_{\dual{C_{\mathbf{v}}}}[l,l,l]$.
Next, we encode $\card{B_3(\mathbf{v},0)}$ for all $\mathbf{v} \in \ZZ_3^n$ into
the weight enumerator of a certain code by the following lemma.
\begin{lemma}
	\label{lemma:duality}
Let $(s_0,s_1,s_2)$ and $(t_0,t_1,t_2)$ be two types of vectors in $\ZZ_3^n$.
Then
\begin{equation}
	\label{eqn:duality}
\binom{n}{s_0,s_1,s_2} \cdot A_{\dual{(s_0,s_1,s_2)}}[t_0,t_1,t_2] =
\binomial{n}{t_0,t_1,t_2} \cdot A_{\dual{(t_0,t_1,t_2)}}[s_0,s_1,s_2].
\end{equation}
\end{lemma}
\begin{proof}
Consider the following bipartite graph on the vertices $\ZZ_3^n(s_0,s_1,s_2)
\cup \ZZ_3^n(t_0,t_1,t_2)$, where $\mathbf{v} \in \ZZ_3^n(s_0,s_1,s_2)$ and
$\mathbf{w} \in \ZZ_3^n(t_0,t_1,t_2)$ are adjacent if and only if $\mathbf{v}
\cdot \mathbf{w} = 0$. The degree of $\mathbf{v}$ in this bipartite graph is
exactly $A_{\dual{(s_0,s_1,s_2)}}[t_0,t_1,t_2]$. Similarly, the degree of
$\mathbf{w}$ is exactly $A_{\dual{(t_0,t_1,t_2)}}[s_0,s_1,s_2]$. Therefore, we
have that
$$
\card{\ZZ_3^n(s_0,s_1,s_2)} \cdot A_{\dual{(s_0,s_1,s_2)}}[t_0,t_1,t_2]
= \card{\ZZ_3^n(t_0,t_1,t_2)} \cdot A_{\dual{(t_0,t_1,t_2)}}[s_0,s_1,s_2].
$$
The proof is completed by noting that $\card{\ZZ_3^n(s_0,s_1,s_2)} =
\binom{n}{s_0,s_1,s_2}$ and $\card{\ZZ_3^n(t_0,t_1,t_2)} =
\binom{n}{t_0,t_1,t_2}$.
\end{proof}
Now, with \autoref{lemma:duality},
\begin{equation}
	\label{eqn:size_balanced_v_dual}
\card{B_3(\mathbf{v},0)} = A_{\dual{T(\mathbf{v})}}[l,l,l]
= \frac{\binom{n}{l,l,l}}{\binom{n}{T(\mathbf{v})}}
A_{\dual{(l,l,l)}}[T(\mathbf{v})].
\end{equation}
Note that $A_{(l,l,l)}(x,y,z) = x^n + 2(xyz)^{n/3}$. Apply the MacWilliams
transform (\autoref{eqn:MacWilliams}),
\begin{equation}
	\label{eqn:weight_enumerator_dual_balanced}
A_{\dual{(l,l,l)}}(x,y,z) = \frac{1}{3}
\qty(
(x+y+z)^n + 2(x^3+y^3+z^3-3xyz)^{n/3}
).
\end{equation}
Finally, by \autoref{eqn:eigenvalue3_code}, \autoref{eqn:size_balanced_v_dual}
and \autoref{eqn:weight_enumerator_dual_balanced}, we conclude that
\begin{align}
\lambda({\mathbf{v}}) &=
\frac{1}{2}\frac{\binom{n}{l,l,l}}{\binomial{n}{T(\mathbf{v})}}
\qty((x+y+z)^n+2(x^3+y^3+z^3-3xyz)^{n/3})[T(\mathbf{v})] - \frac{1}{2}\card{B_3} \\
&= \frac{\binom{n}{l,l,l}}{\binomial{n}{T(\mathbf{v})}}
(x^3+y^3+z^3-3xyz)^{n/3}[T(\mathbf{v})]. \label{eqn:lambda_v}
\end{align}
From \autoref{eqn:lambda_v}, $\lambda({\mathbf{v}})$ depends only on the type
of $\mathbf{v}$, so we write
\begin{equation}
	\label{eqn:lambda_T}
	\lambda(T(\mathbf{v})) \defeq \lambda({\mathbf{v}}).
\end{equation}
Moreover, as the polynomial in \autoref{eqn:lambda_v} is symmetric, for any type
$(t_0,t_1,t_2)$ of vectors in $\ZZ_3^n$,
$\lambda(t_0,t_1,t_2)$ is independent of the order of $t_0,t_1,t_2$. So we can always
assume that $t_0 \le t_1 \le t_2$.

\section{Eigenvalues of $O_{3l,3}$ with Large Absolute Values}
\label{sec:3}
Let $n=3l$. It is easy to check that  $\lambda(0,0,n) = \binom{n}{l,l,l}$ is the
largest eigenvalue of $O_{n,3}$. Further, $\lambda(0,0,n)$ is of the largest
absolute value among all eigenvalues of $O_{n,3}$. In this section, we show that
for each $n=3l$,  $\lambda(1,1,n-2)$ is negative and with the second largest
absolute value, and thus it is the smallest eigenvalue of $O_{n,3}$;
$\lambda(2,2,n-4)$ is positive and with the third largest absolute value, thus
it is the second largest eigenvalue of $O_{n,3}$.

Note that during the preparation of this paper,  an asymptotic version of
\autoref{theorem:goal} was established in the recent online paper \cite{SDU}. In
fact, in \cite{SDU}, the second largest absolute value of
$\cay(\ZZ_p^{lp},\ZZ_p^{lp}(l,l,\dots,l))$ was determined for $l$ large enough
with respect to $p$ and $(p-1)l$ even, where $p \ge 2$ is any positive integer.
As our result is for any $n$ divided by $3$ when $p=3$,  we decided to reserve a
place for \autoref{theorem:goal} in this paper.

\begin{theorem}
	\label{theorem:goal}
Let $n = 3l$ and $(t_0,t_1,t_2)$ be a type of vectors in $\ZZ_3^n$ such that $0
\le t_0 \le t_1 \le t_2 \le n$ and $(t_0,t_1,t_2) \neq (0,0,n)$. Then,
\begin{equation}
	\label{eqn:abt012}
\abs{\lambda(t_0,t_1,t_2)} \leq  \abs{\lambda(1,1,n-2)}.
\end{equation}
In particular, $\lambda(1,1,n-2)=-\frac{\binom{n}{l,l,l}}{n-1}$ is the smallest
eigenvalue of $O_{n,3}$.
\end{theorem}
\begin{proof}
By \autoref{eqn:lambda_v} and \autoref{eqn:lambda_T}, we have
\begin{equation}
	\label{eqn:t012}
\lambda(t_0,t_1,t_2) = \frac{\binom{n}{l,l,l}}{\binom{n}{t_0,t_1,t_2}}
\qty(x^3+y^3+z^3-3xyz)^l[t_0,t_1,t_2].
\end{equation}
Especially,
\begin{equation}
	\label{eqn:lambda_smallest}
\lambda(1,1,n-2)
=\frac{\binom{n}{l,l,l}}{\binom{n}{1,1,n-2}}(-3l)=-\frac{\binom{n}{l,l,l}}{n-1}<0.
\end{equation}
It is clear that
\begin{equation}
\label{eqn:ubt012}
\abs{\lambda(t_0,t_1,t_2)} \le \frac{\binom{n}{l,l,l}}{\binom{n}{t_0,t_1,t_2}}
\qty(x^3+y^3+z^3+3xyz)^l[t_0,t_1,t_2],
\end{equation} with equality holds when $t_0\leq 2$.
Denote $h(t_0,t_1,t_2) \defeq \qty(x^3+y^3+z^3+3xyz)^l[t_0,t_1,t_2]$. If  $t_0 \equiv t_1 \equiv t_2 \pmod{3}$ does not hold, $h(t_0,t_1,t_2)=\lambda(t_0,t_1,t_2)=0$.
So it suffices to show that
\begin{equation}
\label{eqn:ub_abs_lambda_le}
h(t_0,t_1,t_2)\leq \frac{\binom{n}{t_0,t_1,t_2}}{n-1}
\end{equation} when $t_0 \equiv t_1 \equiv t_2 \pmod{3}$.
We will prove \autoref{eqn:ub_abs_lambda_le} by induction on $n$. The base cases
when $t_0\leq 2$ can be obtained by \autoref{cl1} combining the fact that the
equality in \autoref{eqn:ubt012} holds when $t_0\leq 2$.

\begin{claim}\label{cl1} 
When $t_0\leq 2$, $t_0 \equiv t_1 \equiv t_2 \pmod{3}$,
$t_0\leq t_1\leq t_2$ and $(t_0,t_1,t_2)\neq (0,0,n)$,
$\abs{\lambda(t_0,t_1,t_2)} \leq  \abs{\lambda(1,1,n-2)}$.
\end{claim}
\begin{proof}We only prove the case when $t_0=0$, while the proofs for $t_0=1,2$ are similar and can be found in  \autoref{appsec:cl1}.

Assume that  $t_0=0$, $t_1 \equiv t_2\equiv 0\pmod{3}$ and $0< t_1\leq t_2$.
By \autoref{eqn:t012}, $\lambda(0,t_1,t_2)=\frac{\binom{n}{l,l,l}}{\binom{n}{t_1}}\binom{l}{t_1/3}$. Then $\abs{\lambda(0,t_1,t_2)} \le
\abs{\lambda(1,1,n-2)}$ is equivalent to that
$
n-1 \le \frac{\binom{n}{t_1}}{\binom{l}{t_1/3}}.
$
Let $f(t) \defeq \frac{\binom{n}{t}}{\binom{l}{t/3}}$, then for
$t \equiv 0 \pmod{3}$ such that $0 < t$ and $t+3 \le n/2$,
$$
\frac{f(t+3)}{f(t)} = \frac{(n-t-1)(n-t-2)}{(t+2)(t+1)} \ge 1.
$$
Therefore, $f(t)$ is increasing with respect to $t$, and
$
f(t) \ge f(3) = \frac{(n-1)(n-2)}{2} \ge n-1
$
for $n \ge 4$.
\end{proof}

Now assume that  $t_0\geq 3$, $t_0\equiv t_1 \equiv t_2
\pmod{3}$  and $t_0\leq t_1\leq t_2$.
Then
\begin{align*}
h(t_0, t_1, t_2) 
&= 3h(t_0 - 1, t_1 - 1, t_2 - 1) + h(t_0 - 3, t_1, t_2) \\ 
&\qquad + h(t_0, t_1 - 3, t_2) + h(t_0, t_1, t_2 - 3) \\
&\leq \frac{1}{n-4} \left( 3 \binom{n-3}{t_0-1, t_1-1, t_2-1} + \binom{n-3}{t_0-3, t_1, t_2} \right. \\
&\qquad \left. + \binom{n-3}{t_0, t_1-3, t_2} + \binom{n-3}{t_0, t_1, t_2-3} \right),
\end{align*}
where the inequality follows from the induction hypothesis. To get
\autoref{eqn:ub_abs_lambda_le}, it then suffices to have
$$
3t_0t_1t_2 + t_0(t_0-1)(t_0-2) + t_1(t_1-1)(t_1-2) + t_2(t_2-1)(t_2-2) \le (n-4)(n-2)n.
$$
Let
$$
g(t_0,t_1,t_2) \defeq (n-4)(n-2)n - t_0(t_0-1)(t_0-2) - t_1(t_1-1)(t_1-2) - t_2(t_2-1)(t_2-2) - 3t_0t_1t_2,
$$ where $n=t_0+t_1+t_2$.
It suffices to show that $g(t_0,t_1,t_2) \ge 0$ for all $t_0,t_1,t_2 \ge 3$. We think of $t_1$
and $t_2$ as being fixed, and show that $g(t_0,t_1,t_2)$ is increasing with respect to
$t_0$. By calculation
$$
g(t_0,t_1,t_2) - g(t_0-1,t_1,t_2) = 3t_1^2-15t_1 + 3t_2^2-15t_2 + 3t_1t_2 + 9 + t_0(6t_1+6t_2-6).
$$
When $t_0,t_1,t_2 \ge 3$, we have $g(t_0,t_1,t_2) - g(t_0-1,t_1,t_2) \ge 0$.
Thus, $g(t_0,t_1,t_2)$ is increasing with respect to $t_0 \ge 3$, for fixed
$t_1,t_2 \ge 3$. By the symmetry of $t_0,t_1,t_2$, $g(t_0,t_1,t_2)$ is
increasing with respect to $t_0,t_1,t_2 \ge 3$, respectly. In conclusion, for
$t_0,t_1,t_2 \ge 3$,
$
g(t_0,t_1,t_2) \ge g(3,3,3) = 216 > 0.
$
The proof is completed.
\end{proof}

\autoref{theorem:goal} immediately implies that $\lambda(1,1,3l-2)$ is the smallest
eigenvalue of $O_{3l,3}$. Together with the spectral lower bound on quantum chromatic
numbers, $\chi_q(O_{3l,3})$ can be determined for any $l$.
\begin{corollary}
	\label{cor:3l}
$\chi_{q}(O_{3l,3}) = 3l$ for all $l \ge 1$.
\end{corollary}
Before ending this section, we determine the eigenvalue with the third largest
absolute value, which is the second largest eigenvalue of $O_{n,3}$.

\begin{theorem}
\label{theorem:third_largest_eigenvalue_3}
Let $n = 3l \ge 9$ and $(t_0,t_1,t_2)$ be a type of vectors in $\ZZ_3^n$ such that
$0 \le t_0 \le t_1 \le t_2 \le n$ and $(t_0,t_1,t_2) \notin
\setnd{(0,0,n), (1,1,n-2)}$. Then,
$$
\abs{\lambda(t_0,t_1,t_2)} \le \abs{\lambda(2,2,n-4)}.
$$
In particular,
$
\lambda(2,2,n-4) = \frac{2\binom{n}{l,l,l}}{(n-1)(n-2)}
$
is the second largest eigenvalue of $O_{n,3}$.
\end{theorem}
The proof of
\autoref{theorem:third_largest_eigenvalue_3} is similar to that of
\autoref{theorem:goal} and thus moved to
\autoref{appsec:third_largest_eigenvalue_3}. \autoref{theorem:third_largest_eigenvalue_3} will not be used in the rest of
this paper. But we hope that it will be useful for future study.

\section{Quantum Chromatic Numbers of Certain Subgraphs}
\label{sec:subgraphs}
In this section, we determine the quantum chromatic numbers for several classes
of Cayley graphs, which can be considered as subgraphs of  $O_{n,p}$ with
$p=2,3$. We will frequently use \autoref{eq:subg}, that is, if there is a
homomorphism from $G$ to $H$, then  $\chi_q(G) \le \chi_q(H)$.

First, we give a general result on graphs and subgraphs with the same quantum
chromatic number.
\begin{theorem}
	\label{theorem:subgraph_gener}
For any integer $p$ and ordered non-negative $p$-partition
$(t_0,t_1,\ldots,t_{p-1})$ of $n$ with $t_i=t_{p-i}$ for $i\in [\floor{p/2}]$
let $G=\cay(\ZZ_p^{n},\ZZ_p^{n}(t_0,t_1,\ldots,t_{p-1}))$ and
$H=\cay(\ZZ_p^{n-1},S)$ where $S=\ZZ_p^{n-1}(t_0-1,t_1,\ldots,t_{p-1})\cup
\ZZ_p^{n-1}(t_0,t_1-1,\ldots,t_{p-1})\cup \cdots \cup
\ZZ_p^{n-1}(t_0,t_1,\ldots,t_{p-1}-1)$. Suppose $\sum_{i=0}^{p-1}it_i \equiv
0\pmod p$, then \[\chi_q(G) = \chi_q(H).\]
\end{theorem}
\begin{proof}
We prove by constructing homomorphisms between $G$ and $H$. We claim that $H$
can be embedded into $G$ through the following map
$$
\phi : \ZZ_p^{n-1} \to \ZZ_p^{n}, \quad \mathbf{x} \mapsto
\left(\mathbf{x},-\mathbf{x}\cdot  \mathbf{1}\right),
$$ where $\mathbf{1}$ is the all-one vector in $\ZZ_p^{n-1}$.
In other words, $\phi(\mathbf{x})$ is derived by appending a  check bit to
$\mathbf{x}$. We  prove that $\phi$ is indeed a homomorphism. For any two
vectors $\mathbf{x}$ and $\mathbf{y}$ connected in $H$,
$T(\mathbf{x}-\mathbf{y})=(t_0,t_1,\ldots,t_{j}-1,\ldots,t_{p-1})$ for some
$j\in [0,p-1]$. Then $ (\mathbf{x}-\mathbf{y}) \cdot  \mathbf{1}\equiv
\sum_{i=0}^{p-1}it_i-j \equiv -j \pmod p$. Since
$\phi(\mathbf{x})=(\mathbf{x},-\mathbf{x}\cdot  \mathbf{1})$ and
$\phi(\mathbf{y})=(\mathbf{y},-\mathbf{y}\cdot  \mathbf{1})$,
$\phi(\mathbf{x})-\phi(\mathbf{y})=(\mathbf{x}-\mathbf{y},(\mathbf{y}-\mathbf{x})
\cdot  \mathbf{1})=(\mathbf{x}-\mathbf{y},j)$. So
$T(\phi(\mathbf{x})-\phi(\mathbf{y}))=(t_0,t_1,\ldots,t_{j},\ldots,t_{p-1})$,
which means that $\phi(\mathbf{x})$ and $\phi(\mathbf{y})$ are connected in $G$.
This shows that $\chi_q(H) \le \chi_q(G)$.

Conversely, the projection onto the first $n-1$ coordinates is a homomorphism of
graphs $G \to H$, i.e.,
$$
\ZZ_p^{n} \to \ZZ_p^{n-1}, \quad (x_1,x_2,\dots,x_{n-1},x_{n}) \mapsto
(x_1,x_2,\dots,x_{n-1}).
$$
This shows that $\chi_q(G) \le \chi_q(H)$.
\end{proof}

Based on previous results on quantum chromatic numbers in \autoref{tab:results},
we have the following corollary. Recall that $L_r$ is the set of vectors of
weight $r$.

\begin{corollary}
	\label{corollary:easy}
\begin{itemize}
\item[(1)] For any $t\geq 1$, $\chi_q(\mathcal{G}_1) = \chi_q(\mathcal{G}_2)
=4t$, where $$\mathcal{G}_1\triangleq \cay(\ZZ_2^{4t-1},L_{2t-1} \cup L_{2t})
\text{ and } \mathcal{G}_2\triangleq \cay(\ZZ_2^{4t-2},L_{2t-1} \cup L_{2t}).$$
\item[(2)] When $l \ge 1$, $\chi_q(\mathcal{G}_3) = 3l$, where
\[\mathcal{G}_3 \defeq \cay(\ZZ_3^{3l-1},\ZZ_3^{3l-1}(l-1,l,l) \cup
\ZZ_3^{3l-1}(l,l-1,l) \cup \ZZ_3^{3l-1}(l,l,l-1)).\]
\item[(3)] When $p\geq 4$ is an integer, $l$ is large enough and $l(p-1)$ is
even, $\chi_q(\mathcal{G}_4) = lp$, where
\[\mathcal{G}_4 \defeq \cay(\ZZ_p^{lp-1},\ZZ_p^{lp-1}(l-1,l,\ldots,l) \cup
\ZZ_p^{lp-1}(l,l-1,\ldots,l) \cup \dots \cup\ZZ_p^{lp-1}(l,l,\ldots,l-1)).\]
\end{itemize}
\end{corollary}

\subsection{Subgraphs of $O_{n,3}$} Consider the Cayley graph
\[\mathcal{G}_5\triangleq \cay(\ZZ_3^{3l-1},\ZZ_3^{3l-1}(l-1,l,l)), \] that is
the Cayley graph over $\ZZ_3^{3l-1}$ generated by vectors of type $(l-1,l,l)$.
Then $\mathcal{G}_5$ can be embedded into $O_{3l,3} \cong
\cay(\ZZ_3^{3l},\ZZ_3^{3l}(l,l,l))$ by the following homomorphism
\begin{equation}
	\label{eqn:homomorphism_3}
\ZZ_3^{3l-1} \to \ZZ_3^{3l}, \quad \mathbf{x} \mapsto (\mathbf{x},0).
\end{equation}
This shows that $\chi_q(\mathcal{G}_5) \le\chi_q(O_{3l,3}) =3l.$
Next, we determine the smallest eigenvalue of $\mathcal{G}_5$.
\begin{theorem}
	\label{theorem:subgraph_eigenvalue_3}
When $l\geq 2$, the smallest eigenvalue of $\mathcal{G}_5=
\cay(\ZZ_3^{3l-1},\ZZ_3^{3l-1}(l-1,l,l))$ is given by
$$
\lambda_{\min}(\mathcal{G}_5) = \lambda(0,1,3l-2)
= -\frac{3l-2}{l^2}\binomial{3l-3}{l-1,l-1,l-1}.
$$
\end{theorem}
\begin{proof}
Suppose $\mathbf{v} \in \ZZ_3^{3l-1}$ is of type $T(\mathbf{v}) =
(t_0,t_1,t_2)$, where $l \ge 2$. Following the approach described in
\autoref{subsec:coding_approach},
the eigenvalue $\lambda(\mathbf{v})$ of $\mathcal{G}_5$ is given by
$$
\lambda(\mathbf{v}) = \frac{\binomial{3l-1}{l-1,l,l}}{\binomial{3l-1}{t_0,t_1,t_2}}
\qty((x^3+y^3+z^3-3xyz)^{l-1}(x^2+y^2+z^2-xy-yz-xz))[t_0,t_1,t_2].
$$
In the following, we illustrate how to determine the smallest eigenvalue. To
distinguish between the eigenvalues of $\mathcal{G}_5$ and the eigenvalues of
$O_{3l,3}$, we denote by $\lambda(t_0,t_1,t_2) \defeq \lambda(\mathbf{v})$ the
eigenvalues of $\mathcal{G}_5$ as in the statement of the theorem, and
$\lambda_{O_{3l-3,3}}(t_0,t_1,t_2)$ the eigenvalues of $O_{3l-3,3}$.

For simplicity, denote $W(x,y,z) \defeq (x^3+y^3+z^3-3xyz)^{l-1}$, $Y(x,y,z)
\defeq x^2+y^2+z^2-xy-yz-xz$ and write  $T = T(\mathbf{v}) = (t_0,t_1,t_2)$.
Assume that $(WY)[T] \neq 0$. Then by the formulas of $W$ and $Y$, there is a
unique pair of nonnegative ordered  $3$-partitions $(T_0,T_1)$, where $T_0$ is a
partition of $3l-3$ and  $T_1$ is a partition of $2$, such that $T = T_0 + T_1$
(here is the vector addition) and $(WY)[T] = W[T_0] \cdot Y[T_1]$. Further, the
components of $T_0$ are congruent modulo $3$ and
$$
T_1 \in \setnd{(2,0,0),(0,2,0),(0,0,2),(1,1,0),(1,0,1),(0,1,1)}.
$$

As $\lambda(t_0,t_1,t_2)$ is independent of the order of $t_0,t_1,t_2$, we only
need to consider the following two cases that
$$
T_1 \in \setnd{(1,1,0), (2,0,0)}.
$$
If $T_1 = (1,1,0)$, we have
\begin{align}
\lambda(t_0,t_1,t_2) &= -\frac{\binomial{3l-1}{l-1,l,l}}{\binomial{3l-1}{t_0,t_1,t_2}}
\frac{\binom{3l-3}{t_0-1,t_1-1,t_2}}{\binomial{3l-3}{l-1,l-1,l-1}}
\lambda_{O_{3l-3,3}}(t_0-1,t_1-1,t_2) \label{eq:express1} \\ 
&= -\frac{t_0t_1}{l^2}\lambda_{O_{3l-3,3}}(t_0-1,t_1-1,t_2). \label{eq:express2}
\end{align}
We show that $\lambda(3l-2,1,0)$ is the smallest eigenvalue in this case.
The proof is by contradiction. Assume that
$\lambda(t_0,t_1,t_2) < \lambda(3l-2,1,0)$ for some $(t_0,t_1,t_2)$ in this case.
Then by \autoref{eq:express1} and \autoref{eq:express2},
\begin{equation}
-\frac{t_0t_1}{l^2}\lambda_{O_{3l-3,3}}(t_0-1,t_1-1,t_2) <
-\frac{3l-2}{l^2}\lambda_{O_{3l-3,3}}(3l-3,0,0).
\end{equation}
This implies that
\begin{equation}\label{eq:35}
\frac{t_0t_1}{3l-2} > \frac{\lambda_{O_{3l-3,3}}(3l-3,0,0)}{\lambda_{O_{3l-3,3}}(t_0-1,t_1-1,t_2)}
\ge \frac{\lambda_{O_{3l-3,3}}(3l-3,0,0)}{\abs{\lambda_{\min}(O_{3l-3,3})}} =  3l-4,
\end{equation}
where the first inequality follows from the fact that $\lambda(t_0,t_1,t_2)$ is negative
and thus $\lambda_{O_{3l-3,3}}(t_0-1,t_1-1,t_2)$ is positive, and
the second inequality and the last equality follow from \autoref{theorem:goal} and
\autoref{eqn:lambda_smallest}.
So we have $t_0t_1 > (3l-2)(3l-4)$, which contradicts to the assumption that
$t_0+t_1+t_2 = 3l-1$ and $l \ge 2$.

If $T_1 =(2,0,0)$, we have
\begin{align}
\lambda(t_0,t_1,t_2) &= \frac{\binomial{3l-1}{l-1,l,l}}{\binomial{3l-1}{t_0,t_1,t_2}}
\frac{\binom{3l-3}{t_0-2,t_1,t_2}}{\binomial{3l-3}{l-1,l-1,l-1}}
\lambda_{O_{3l-3,3}}(t_0-2,t_1,t_2) \\
&= \frac{t_0(t_0-1)}{l^2}\lambda_{O_{3l-3,3}}(t_0-2,t_1,t_2).
\end{align}
Again, we show that $\lambda(t_0,t_1,t_2)$ cannot be  smaller than
$\lambda(3l-2,1,0)$. Otherwise, if $\lambda(t_0,t_1,t_2) < \lambda(3l-2,1,0)$, we have
\begin{equation}
\frac{t_0(t_0-1)}{l^2}\lambda_{O_{3l-3,3}}(t_0-2,t_1,t_2) < -\frac{3l-2}{l^2}\lambda_{O_{3l-3,3}}(3l-3,0,0).
\end{equation}
Similar to \autoref{eq:35}, we have 
\begin{equation}
\frac{t_0(t_0-1)}{3l-2} > \frac{\lambda_{O_{3l-3,3}}(3l-3,0,0)}{\abs{\lambda_{O_{3l-3,3}}(t_0-2,t_1,t_2)}}
\ge \frac{\lambda_{O_{3l-3,3}}(3l-3,0,0)}{\abs{\lambda_{\min}(O_{3l-3,3})}} = 3l-4,
\end{equation}
which yields the contradiction  $t_0(t_0-1) > (3l-2)(3l-4)$.
\end{proof}
\begin{theorem}\label{g5}
When $l \ge 2$, $\chi_q(\mathcal{G}_5) = 3l$, where $\mathcal{G}_5=\cay(\ZZ_3^{3l-1},\ZZ_3^{3l-1}(l-1,l,l))$.
\end{theorem}
\begin{proof}
The result follows immediately from the spectral lower bound in \autoref{eqn:spectral_lb}
and \autoref{theorem:subgraph_eigenvalue_3}.
\end{proof}

Applying \autoref{theorem:subgraph_gener} to \autoref{g5}, we have the following corollary.

\begin{corollary}
	\label{theorem:subgraph_3_III}
When $l \ge 2$, $\chi_q(\mathcal{G}_6) = 3l$, where
\[\mathcal{G}_6 \defeq
\cay(\ZZ_3^{3l-2},\ZZ_3^{3l-2}(l-2,l,l) \cup \ZZ_3^{3l-2}(l-1,l-1,l) \cup \ZZ_3^{3l-2}(l-1,l,l-1)).\]
\end{corollary}

\section{Quantum Chromatic Number of $\cay(\ZZ_2^n,L_2)$}
\label{sec:Hn2}
Denote $H(n,r) \defeq \cay(\ZZ_2^n,L_r)$. When $r$ is odd, $H(n,r)$ is bipartite
thus $\chi_q(H(n,r))=\chi(H(n,r))=2$. For even $r \ge n/2$, it was shown in
\cite{Feng} that
$$
r \le \chi_q(H(n,r)) \le 2r.
$$
The open problem to find a good upper bound on $\chi_q(H(n,r))$ for even $r \le
n/2$ was raised in \cite{Feng}. In this section, as a step toward this problem,
we focus on the case $r = 2$ and determine $\chi_q(H(n,2))$ for infinitely many
$n$ by the theory of combinatorial designs.

We first compute the spectral lower bound of $\chi_q(H(n,2))$. As in
\autoref{subsecp2}, all different eigenvalues of $H(n,2)$ are
$$
\lambda(w) = K_2(w) = (n-2w)^2/2 - n/2, ~w\in [n].
$$
Thus,
$$
\lambda_{\min}(H(n,2)) =
\begin{cases}
-n/2, & n \equiv 0 \pmod{2}, \\
-(n-1)/2, & n \equiv 1 \pmod{2}.
\end{cases}
$$
Clearly, $\lambda_{\max}(H(n,2)) = \binom{n}{2}$. So the spectral lower bound
shows that
\begin{equation}
	\label{eqn:spectral_lb_Hn2}
\chi_q(H(n,2)) \ge 1-\frac{\lambda_{\max}(H(n,2))}{\lambda_{\min}(H(n,2))} =
\begin{cases}
n, & n \equiv 0 \pmod{2}, \\
n+1, & n \equiv 1 \pmod{2}.
\end{cases}
\end{equation}

To find an upper bound on $\chi_q(H(n,2))$, we use the theory of combinatorial designs
to construct flat orthogonal representations of $H(n,2)$. To this end, we introduce the
following combinatorial structure. Let $2^{[n]}$ denote the set of all subsets of $[n]$.
\begin{definition}[balanced pair-separating family]
	\label{def:design}
A family $\mathcal{F}\subset 2^{[n]}$ is called a \emph{balanced pair-separating
family with separation number $\theta$} if for every pair of distinct elements
$i,j \in [n]$, there are exactly $\theta$ members (called \emph{blocks}) of
$\mathcal{F}$ containing only one of $i,j$.
\end{definition}

The following theorem shows that a balanced pair-separating family
with a big separation number can help to construct flat orthogonal
representations of $H(n, 2)$, and hence provide an upper bound of
$\chi_q(H(n,2))$. Note that there is a natural one to one correspondence between
vectors in $\ZZ_2^n$ and subsets of $[n]$ by defining $\supp (\mathbf{v}) \defeq
\set{i \in [n]}{v_i \neq 0}$ for each  $\mathbf{v}\in \ZZ_2^n$. Through this
correspondence, we will not distinguish vectors in $\ZZ_2^n$ and subsets of
$[n]$.
\begin{theorem}
	\label{theorem:design}
Suppose there is a balanced pair-separating  family $\mathcal{F} \subset
2^{[n]}$ with separation number $\theta \ge |\mathcal{F}|/2$. Then
$\chi_q(H(n,2))\leq 2\theta$.

\end{theorem}
\begin{proof}Let $\mathcal{F}=\{B_1,B_2,\dots,B_b\}$. Then $2\theta\geq b$.
Consider the following map
$$
\phi : \ZZ_2^n \to \setnd{\pm 1}^{2\theta}, \quad \mathbf{v} \mapsto
(\chi_1(\mathbf{v}),\chi_2(\mathbf{v}),\dots,\chi_b(\mathbf{v}),\mathbf{1}_{2\theta-b}),
$$
where $\chi_i \defeq \chi_{B_i}$ is the character of $\ZZ_2^n$ corresponding to
$B_i$, and $\mathbf{1}_{2\theta-b}$ is the all-one vector of length $2\theta-b$.
It suffices to prove that $\phi$ is a flat orthogonal representation of $H(n,2)$
of dimension $2\theta$.

Given adjacent $\mathbf{v}$ and $\mathbf{w}$ in $H(n,2)$, let $\mathbf{z} =
\mathbf{v} + \mathbf{w}$. Then $\wt(\mathbf{z})=2$. We need to show that
$\phi(\mathbf{v})$ and $\phi(\mathbf{w})$ are orthogonal. Consider the inner
product
$$
\inp{\phi(\mathbf{v})}{\phi(\mathbf{w})} = \sum_{i=1}^b \chi_i(\mathbf{v})\chi_i(\mathbf{w})
+ \qty(2\theta - b) = \sum_{i=1}^b \chi_i(\mathbf{z}) + \qty(2\theta - b).
$$
Note that
$\chi_i(\mathbf{z}) = -1$ if and only if $\card{\supp (\mathbf{z}) \cap B_i} = 1$. Therefore,
$$
\inp{\mathbf{v}}{\mathbf{w}} = b-2\theta + 2\theta - b  = 0.
$$
As $\phi(\mathbf{v})$ consists of only $\pm 1$, $\phi$ is flat.
\end{proof}

Next, we show the existence of  balanced pair-separating families from balanced
incomplete block designs (BIBD). An $(n,k,\lambda)$-BIBD is a family
$\mathcal{B}$ of $k$-subsets of $[n]$, called blocks, such that any pair of
distinct elements in $[n]$ is contained in exactly $\lambda$ blocks
\cite[Chapter 1]{Stinson}. In such a BIBD, each element of $[n]$ is contained in
exactly $r = \frac{\lambda (n-1)}{k-1}$ blocks  and the number of blocks is
$$
b \triangleq|\mathcal{B}|=\frac{nr}{k} = \frac{\lambda n(n-1)}{k(k-1)}.
$$

By definition, it is easy to see that if $\mathcal{B}$ is an
$(n,k,\lambda)$-BIBD, then $\mathcal{B}$ is a balanced pair-separating family
with separation number $\theta = 2(r-\lambda)$. By the requirement $\theta \ge
|\mathcal{F}|/2$ in \autoref{theorem:design}, and combining formulas on $r$ and
$b$, we have the following observation.

\begin{observation}\label{obs} 
If there is an $(n,k,\lambda)$-BIBD with
$4k(n-k)\geq n(n-1)$, then $$\chi_q(H(n,2))\leq \frac{4\lambda(n-k)}{k-1}.$$
\end{observation}

In order to get upper bounds matching the spectral lower bounds in
\autoref{eqn:spectral_lb_Hn2}, BIBDs with the following parameters are useful.

\begin{corollary}
\label{lemma:symmetric_BIBD}
\begin{itemize}
\item[(1)] If there exists an $(n,(n-1)/2,(n-3)/4)$-BIBD with $n \equiv 3
\pmod{4}$, then $\chi_q(H(n,2)) = n+1$.
\item[(2)]  If there is a $(4s^2,2s^2-s,s^2-s)$-BIBD for some integer $s$, then
$\chi_q(H(4s^2,2)) = 4s^2$.
\end{itemize}
\end{corollary}
\begin{proof} The upper bounds are from \autoref{obs} and the lower bounds are
from \autoref{eqn:spectral_lb_Hn2}.
\end{proof}

BIBDs with parameters in \autoref{lemma:symmetric_BIBD} are in fact
symmetric BIBDs, i.e., satisfying $b=n$.
In \autoref{tab:BIBD}, we collect some known existence of symmetric
$(n,k,\lambda)$-BIBDs satisfying the requirements of \autoref{lemma:symmetric_BIBD}
and determine the corresponding quantum chromatic of $H(n,2)$.

\begin{table}
	\centering
	\caption{Exact Quantum Chromatic Number of $H(n,2)$ from Symmetric BIBD}
	\label{tab:BIBD}
	\begin{tabular}{lcl}
	\toprule
	$(n=b,k=r,\lambda)$ & 
	$\chi_q(H(n,2))$ &
	Comments and References \\
	\midrule
	\multirow{2}{*}{$(q,(q-1)/2,(q-3)/4)$} &
	\multirow{2}{*}{$q+1$} &
	\multirow{2}{*}{\makecell[l]{$q \equiv 3 \pmod{4}$ is a prime \\ power 
	\cite[Chapter 3]{Stinson}}}\\
	&& \\
	\hline
	\multirow{2}{*}{$(2^{t+2}-1,2^{t+1}-1,2^t-1)$} &
	\multirow{2}{*}{$2^{t+2}$} &
	\multirow{2}{*}{$t$ is a positive integer \cite[II.6.8]{colbourn}} \\
	&&\\
	\hline
	{$(q^2+2q,\frac{q^2+2q-1}{2},\frac{q^2+2q-3}{4})$} &
	{$q^2+2q+1$} &
	\makecell[l]{$q$ and $q+2$ are both odd \\ prime powers 
	\cite[VI.18.5]{colbourn}}\\
	\hline
	\multirow{2}{*}{$(4s^2,2s^2-s,s^2-s)$} & 
	\multirow{2}{*}{$4s^2$} & \multirow{2}{*}{\makecell[l]{Existence of regular
	Hadamard \\ matrix of order $4s^2$ \cite[VI.18.6]{colbourn}}}\\
	&& \\
	\bottomrule
	\end{tabular}
\end{table}

We finish this section with a remark concerning the classical chromatic number
of $H(n,2)$. It is clear that $\chi_q(G) \le \chi(G)$ for any graph $G$. So, it
is natural to consider $\chi(H(n,2))$ as an upper bound of $\chi_q(H(n,2))$. The
values of $\chi(H(n,2))$ has been studied in
\cite{Xing2013coloring,Kokkala2018Chromatic}. For $n \le 16$, we collect known
results on $\chi(H(n,2))$ in the literature
\cite{Kokkala2018Chromatic,Xing2013coloring,Ziegler,Lauri14colorable,chromatic_table}
in \autoref{tab:chi_Hn2}. We can see from \autoref{tab:chi_Hn2} that
$\chi(H(n,2))$ does not match the spectral lower bound of $\chi_q(H(n,2))$ in
general, leaving $\chi_q(H(n,2))$ undetermined. Also from \autoref{tab:chi_Hn2},
for $n \in \setnd{9,10}$, $\chi_q(H(n,2)) \le 12$ while $13 \le \chi(H(n,2))$
and for $n = 11$, $\chi_q(H(11,2)) = 12$ while $15 \le \chi(H(11,2))$,
illustrating a separation between the quantum and classical chromatic numbers of
$H(9,2),H(10,2)$ and $H(11,2)$.

\begin{table}
\centering
\begin{threeparttable}
\caption{Chromatic Number of $H(n,2)$
\cite{Kokkala2018Chromatic,Xing2013coloring,Ziegler,Lauri14colorable,chromatic_table}}
\label{tab:chi_Hn2}
\begin{tabular}{c|cccccccc}
\toprule
$n$ & $2$ & $3$ & $4$ & $5$ & $6$ & $7$ & $8$  & $9$\\
\midrule
$\chi(H(n,2))$ & $2$ & $4$ & $4$ & $8$ & $8$ & $8$ & $8$ & $13$ \\
$\chi_q(H(n,2))$ & $2$ & $4$ & $4$ & $6-8$ & $6-8$ & $8$ & $8$ & $10-12$ \\
\midrule
$n$ & $10$ & $11$ & $12$ & $13$ & $14$ & $15$ & $16$ \\
\midrule
$\chi(H(n,2))$  & $13-14$ & $15-16$ & $15-16$ & $16$ & $16$ & $16$ & $16$ \\
$\chi_q(H(n,2))$ & $10-12$ & $12$ & $12-16$ & $14-16$ & $14-16$ & $16$ & $16$ \\
\bottomrule
\end{tabular}
\begin{tablenotes}
\item[1] References on $\chi(H(n,2))$: $n \le 8$ in \cite{Ziegler},
$n \in [9,10]$  in \cite{Kokkala2018Chromatic,Lauri14colorable},
$n \in [11,16]$  in \cite{chromatic_table}.
\item[2] The lower bounds on $\chi_q(H(n,2))$ is the spectral lower bound in
\autoref{eqn:spectral_lb_Hn2}. For $n = 11$, $\chi_q(H(11,2)) = 12$ by
\autoref{tab:BIBD}. Note that $H(n,2)$ can be embedded into $H(n+1,2)$ for all
$n$. Thus, $\chi_q(H(9,2)) \le \chi_q(H(10,2)) \le \chi_q(H(11,2)) = 12$.
For the rest $n$, $\chi_q(H(n,2)) \le \chi(H(n,2))$.
\end{tablenotes}
\end{threeparttable}
\end{table}
\section{Conclusion}
\label{sec:conclusion}
In this paper, we determined the quantum chromatic number of several Cayley
graphs over $\ZZ_p^n$, which can be regarded as subgraphs of the orthogonality
graphs. Further, we proposed an approach to construct flat orthogonal
representations for the graph $H(n,2)$ by using tools in design theory and
determined the quantum chromatic number of $H(n,2)$ for infinitely $n$. All the
main results of this paper have been summarized in \autoref{tab:results}.

There are still many interesting questions that might be considered in the
future. We list some of them below.
\begin{itemize}
\item[(1)] Can one extend \autoref{g5} into general integer $p$?
\item[(2)] Determine the quantum chromatic number of $H(n,2)$ for more $n$.
\item[(3)] Can one construct balanced pair-separating families with big
separation number not from BIBDs?
\item[(4)]  The definition of balanced pair-separating families can be easily
extended to construct the flat orthogonal representations for $H(n,2t)$.
However, even for $t=2$, we failed to find any known $4$-designs in the
literature yielding an upper bound of $\chi_q(H(n,4))$ matching the spectral
lower bound. Some new ideas are needed to determine the quantum chromatic number
of $H(n,4)$.
\end{itemize}

\subsection*{Acknowledgements}
Thanks to Professor Jack Koolen for helpful discussions on the proof of \autoref{theorem:goal}.
The research of all authors was supported in part
by the Quantum Science and Technology-National Science and Technology Major Project
under Grant 2021ZD0302902, in part by NSFC under Grant 12171452 and Grant 12231014, and in
part by the National Key Research and Development Programs of China under Grant
2023YFA1010200 and Grant 2020YFA0713100.

\appendix
\section{Remaining Proof of \autoref{cl1}}
\label{appsec:cl1}
\begin{proof}
When $t_0=1$, $t_1 \equiv t_2 \equiv 1\pmod{3}$ and $1\leq t_1\leq t_2$, by \autoref{eqn:t012}
$$\lambda(1,t_1,t_2)=-n\frac{\binom{n}{l,l,l}}{\binom{n}{1,t_1,t_2}}\binom{l-1}{(t_1-1)/3}.$$
Then $\abs{\lambda(1,t_1,t_2)} \leq
\abs{\lambda(1,1,n-2)}$ is equivalent to that
$
n-1 \le \frac{\binom{n-1}{t_1}}{\binom{l-1}{(t_1-1)/3}}.
$
Let $f(t) \defeq
\frac{\binom{n-1}{t}}{\binom{l-1}{(t-1)/3}}$, then
for $t \equiv
1 \pmod{3}$ such that $1 \le t$ and $t+3 \le (n-1)/2$,
$$
\frac{f(t+3)}{f(t)} = \frac{(n-t-1)(n-t-3)}{(t+2)(t+1)} \ge 1.
$$
Therefore, $f(t)$ is increasing with respect to $t$, and
$
f(t) \ge f(1) = n-1.
$

When $t_0=2$, $t_1 \equiv t_2
\equiv 2\pmod{3}$  and $2\leq t_1\leq t_2$,
$\lambda(2,t_1,t_2)=9\frac{\binom{n}{l,l,l}}{\binom{n}{2,t_1,t_2}}
\binom{l}{2}\binom{l-2}{(t_1-2)/3}$ by \autoref{eqn:t012}. Then
$\abs{\lambda(2,t_1,t_2)} \leq \abs{\lambda(1,1,n-2)}$ is equivalent to that
$
n-3 \le \frac{\binom{n-2}{t_1}}{\binom{l-2}{(t_1-2)/3}}.
$
Write $k \defeq (t_1-2)/3$, then $k \in [0,\floor{n/6}-1]$.
Let $f(k) \defeq \frac{\binom{n-2}{3k+2}}{\binom{l-2}{k}}$, then for
$0 \le k$ and $k+1 \le \floor{n/6}-1$,
$$
\frac{f(k+1)}{f(k)} = \frac{(n-3k-4)(n-3k-5)}{(3k+4)(3k+5)} \ge 1.
$$
Therefore, $f(k)$ is increasing with respect to $k$, and
$
f(k) \ge f(0) = \frac{(n-2)(n-3)}{2}
\ge n-3
$
for $n \ge 6$.
\end{proof}

\section{Proof of \autoref{theorem:third_largest_eigenvalue_3}}
\label{appsec:third_largest_eigenvalue_3}
\begin{proof}
By \autoref{eqn:lambda_v} and \autoref{eqn:lambda_T},
$$
\lambda(2,2,n-4) = \frac{2\binom{n}{l,l,l}}{(n-1)(n-2)}
$$
is a positive eigenvalue. We show that $\lambda(2,2,n-4)$ is of the third largest
absolute value. The proof follows from a similar approach as in the proof of
\autoref{theorem:goal}. Recall in the proof of \autoref{theorem:goal} that
\begin{equation}
\label{eqn:equality_holds}	
\abs{\lambda(t_0,t_1,t_2)} \le \frac{\binom{n}{l,l,l}}{\binom{n}{t_0,t_1,t_2}}
(x^3+y^3+z^3+3xyz)^l[t_0,t_1,t_2],
\end{equation}
with equality holds when $t_0 \le 2$, and denote
$h(t_0,t_1,t_2) \defeq (x^3+y^3+z^3+3xyz)^l[t_0,t_1,t_2]$.
It suffices to show that
\begin{equation}
\label{eqn:hrst_inequality}	
h(t_0,t_1,t_2) \le \frac{2\binom{n}{t_0,t_1,t_2}}{(n-1)(n-2)}.
\end{equation}
\autoref{eqn:hrst_inequality} will be shown by induction on $n$.
The base case when $t_0 \le 2$ can be obtained by Claims~\ref{appcl1}-\ref{appcl3}
with the fact that the equality in \autoref{eqn:equality_holds} holds when
$t_0 \le 2$.
\begin{claim}
	\label{appcl1}
When $t_0=0$, $t_1 \equiv t_2 \equiv 0\pmod{3}$ and $0< t_1\leq t_2$,
$\abs{\lambda(t_0,t_1,t_2)} \leq \abs{\lambda(2,2,n-4)}.$
\end{claim}
\begin{proof}
By \autoref{eqn:t012},
$\lambda(0,t_1,t_2)=\frac{\binom{n}{l,l,l}}{\binom{n}{t_1}}\binom{l}{t_1/3}$.
Then $\abs{\lambda(0,t_1,t_2)} \le \abs{\lambda(2,2,n-4)}$ is equivalent to that
$\frac{(n-1)(n-2)}{2} \le \frac{\binom{n}{t_1}}{\binom{l}{t_1/3}}$, which has
been verified in the proof of \autoref{cl1}. This essentially shows that
$\lambda(2,2,n-4) = \lambda(0,3,n-3)$ share the same value.
\end{proof}

\begin{claim}
	\label{appcl2}
When $t_0=1$ and $t_1 \equiv t_2 \equiv 1\pmod{3}$ and $1 < t_1 \leq t_2$,
$\abs{\lambda(t_0,t_1,t_2)} \leq \abs{\lambda(2,2,n-4)}$.
\end{claim}
\begin{proof}
By \autoref{eqn:t012},
$\lambda(1,t_1,t_2)=-n\frac{\binom{n}{l,l,l}\binom{l-1}{(t_1-1)/3}}{\binom{n}{1,t_1,t_2}}$
and $\abs{\lambda(1,t_1,t_2)} \leq \abs{\lambda(2,2,n-4)}$ is equivalent to that
$
\frac{(n-1)(n-2)}{2} \le \frac{\binom{n-1}{t_1}}{\binom{l-1}{(t_1-1)/3}}.
$
Let $f(t) \defeq \frac{\binom{n-1}{t}}{\binom{l-1}{(t-1)/3}}$,
for $t \in [4,(n-1)/2]$ with $t \equiv 1 \pmod{3}$.
It has been verified in \autoref{appsec:cl1} that $f(t)$ is increasing with respect to $t$.
Thus, $f(t) \ge f(4) = \frac{(n-1)(n-2)(n-4)}{8} \ge \frac{(n-1)(n-2)}{2}$ when
$n \ge 8$.
\end{proof}

\begin{claim}\label{appcl3}
When $t_0=2$ and $t_1 \equiv t_2
\equiv 2\pmod{3}$  and $2\leq t_1\leq t_2$,
$\abs{\lambda(t_0,t_1,t_2)} \leq  \abs{\lambda(2,2,n-4)}$.
\end{claim}
\begin{proof}
By \autoref{eqn:t012},
$\lambda(2,t_1,t_2)=9\frac{\binom{n}{l,l,l}\binom{l}{2}}
{\binom{n}{2,t_1,t_2}}\binom{l-2}{(t_1-2)/3}$ and 
$\abs{\lambda(2,t_1,t_2)} \leq \abs{\lambda(2,2,n-4)}$ is equivalent to that
$
\frac{(n-2)(n-3)}{2} \le \frac{\binom{n-2}{t_1}}{\binom{l-2}{(t_1-2)/3}},
$
which has been verfied in the proof of \autoref{appsec:cl1}.
\end{proof}
Now assume that $t_0 \ge 3$, $t_0 \equiv t_1 \equiv t_2 \pmod{3}$. Then
\begin{multline*}
h(t_0,t_1,t_2) = h(t_0-3,t_1,t_2) + h(t_0,t_1-3,t_2) + h(t_0,t_1,t_2-3) + 3h(t_0-1,t_1-1,t_2-1) \\
\le \frac{2}{(n-4)(n-5)} \left(
\binom{n-3}{t_0-3,t_1,t_2} + \binom{n-3}{t_0,t_1-3,t_2} + \right. \\
\left. \binom{n-3}{t_0,t_1,t_2-3}  + 3\binom{n-3}{t_0-1,t_1-1,t_2-1} \right),
\end{multline*}
where the inequality follow from the induction hypothesis.
To get \autoref{eqn:hrst_inequality}, it suffices to have
\begin{equation}
	\label{eqn:rst_inequality}
t_0(t_0-1)(t_0-2) + t_1(t_1-1)(t_1-2) + t_2(t_2-1)(t_2-2) + 3t_0t_1t_2 \le n(n-4)(n-5).
\end{equation}
Let $f(t_0,t_1,t_2) \defeq n(n-4)(n-5) - t_0(t_0-1)(t_0-2) - t_1(t_1-1)(t_1-2) -
t_2(t_2-1)(t_2-2) - 3t_0t_1t_2$ where $n = t_0+t_1+t_2$.
It suffices to show that $f(t_0,t_1,t_2) \ge 0$ for all $t_0,t_1,t_2 \ge 3$.
We think of $t_1,t_2$ as fixed and show that $f(t_0,t_1,t_2)$ is increasing with
respect to $t_0$. By calculation
$$
f(t_0,t_1,t_2) - f(t_0-1,t_1,t_2) = 3(t_1^2 - 7t_1 + t_2^2 -7t_2 + t_0t_1 + t_0t_2 + 8) + 6t_0(t_1+t_2-2).
$$
For $t_0,t_1,t_2 \ge 3$, $f(t_0,t_1,t_2) - f(t_0-1,t_1,t_2) \ge 0$ and
$f(t_0,t_1,t_2)$ is increasing with respect to $t_0 \ge 3$. By the symmetry of
$t_0,t_1,t_2$, $f(t_0,t_1,t_2)$ is increasing with respect to $t_0,t_1,t_2 \ge
3$, respectively. In conclusion, for $t_0,t_1,t_2 \ge 3$, $f(t_0,t_1,t_2) \ge
f(3,3,3) = 81 > 0$. The proof is completed.
\end{proof}



\begin{thebibliography}{99}
\bibitem{CostSimulating}
G.~Brassard, R.~Cleve, and A.~Tapp.
\newblock Cost of exactly simulating quantum entanglement with classical communication.
\newblock {\em Physical Review Letters}, 83(9):1874--1877, 1999.

\bibitem{chromatic_table}
A.~E. Brouwer.
\newblock {Online table of chromatic numbers of cubelike hull}.
\newblock Online available at \url{https://aeb.win.tue.nl/graphs/cubelike.html}.
\newblock Accessed on 2025-11-24.

\bibitem{SmallestEigenvalue}
A.~E. Brouwer, S.~M. Cioabă, F.~Ihringer, and M.~McGinnis.
\newblock The smallest eigenvalues of {H}amming graphs, {J}ohnson graphs and other distance-regular graphs with classical parameters.
\newblock {\em Journal of Combinatorial Theory, Series B}, 133:88--121, 2018.

\bibitem{SpectraCayleyGraphs}
A.~E. Brouwer and W.~H. Haemers.
\newblock {\em Spectra of Graphs}.
\newblock Universitext. Springer New York, NY, New York, NY, 2012.

\bibitem{QuantumVSClassical}
H.~Buhrman, R.~Cleve, and A.~Wigderson. \newblock Quantum vs. classical
communication and computation. \newblock In {\em Proceedings of the Thirtieth
Annual ACM Symposium on Theory of Computing}, STOC '98, page 63–68, New York,
NY, USA, 1998. 

\bibitem{OnTheQuantumChromaticNumber}
P.~J. Cameron, A.~Montanaro, M.~W. Newman, S.~Severini, and A.~Winter.
\newblock On the quantum chromatic number of a graph.
\newblock {\em Electronic Journal of Combinatorics}, 14(1), 2007.

\bibitem{SDU}
X.~Cao, K.~Feng, H.~Huang, Y.~Yang, and Z.~Zhang.
\newblock On the quantum chromatic number of {H}amming and generalized {H}adamard graphs.
\newblock {\em arXiv preprint}, arXiv:2510.14209, 2025.

\bibitem{Feng}
X.~Cao, K.~Feng, and Y.~Tan.
\newblock Quantum chromatic numbers of some graphs in {H}amming schemes.
\newblock {\em arXiv preprint}, arXiv:2412.09904, 2024.

\bibitem{colbourn}
C.~J. Colbourn and J.~H. Dinitz, editors.
\newblock {\em Handbook of Combinatorial Designs}.
\newblock Chapman and Hall/CRC, New York, 2nd edition, 2006.

\bibitem{QuantumProtocol}
A.~David, H.~Jun, K.~Yosuke, and S.~Yuuya.
\newblock A quantum protocol to win the graph colouring game on all {H}adamard graphs.
\newblock {\em IEICE Transactions on Fundamentals}, E89-A(5):1378--1381, 2006.

\bibitem{GraphTheory}
R.~Diestel.
\newblock {\em Graph Theory}.
\newblock Graduate Texts in Mathematics. Springer Berlin, Heidelberg, Berlin, Heidelberg, 6 edition, 2024.

\bibitem{SpectralLB}
C.~Elphick and P.~Wocjan.
\newblock Spectral lower bounds for the quantum chromatic number of a graph.
\newblock {\em Journal of Combinatorial Theory, Series A}, 168:338--347, 2019.

\bibitem{Xing2013coloring}
F.~W. Fu, S.~Ling, and C.~P. Xing.
\newblock New results on two hypercube coloring problems.
\newblock {\em Discrete Applied Mathematics}, 161(18):2937--2945, 2013.

\bibitem{JiZhengfeng}
Z.~Ji.
\newblock Binary constraint system games and locally commutative reductions.
\newblock {\em arXiv preprint}, arXiv:1310.3794, 2013.

\bibitem{Kokkala2018Chromatic}
J.~I. Kokkala and P.~R.~J. Östergård.
\newblock The chromatic number of the square of the 8-cube.
\newblock {\em Mathematics of Computation}, 87(313):2551--2561, 2018.

\bibitem{SmallGraphs}
O.~Lalonde.
\newblock On the quantum chromatic numbers of small graphs.
\newblock {\em Electronic Journal of Combinatorics}, 32(1), 2025.

\bibitem{Lauri14colorable}
J.~Lauri.
\newblock The square of the 9-hypercube is 14-colorable.
\newblock {\em arXiv preprint}, arXiv:1605.07613, 2016.

\bibitem{FiniteFields}
R.~Lidl and H.~Niederreiter.
\newblock {\em Finite Fields}.
\newblock Encyclopedia of Mathematics and its Applications. Cambridge University Press, Cambridge, 2 edition, 1996.

\bibitem{Oddities}
L.~Manc{\'i}nska and D.~E. Roberson.
\newblock Oddities of quantum colorings.
\newblock {\em Baltic Journal on Modern Computing}, 4(4):846--859, 2016.

\bibitem{QuantumHomomorphisms}
L.~Manc{\'i}nska and D.~E. Roberson.
\newblock Quantum homomorphisms.
\newblock {\em Journal of Combinatorial Theory, Series B}, 118:228--267, 2016.

\bibitem{ExactHadamardGraphs}
M.~McNamara.
\newblock The exact quantum chromatic number of {H}adamard graphs.
\newblock {\em arXiv preprint}, arXiv:2410.00042, 2024.

\bibitem{SelfDual}
E.~M. Rains and N.~J.~A. Sloane.
\newblock Self-dual codes.
\newblock {\em arXiv preprint}, arXiv:math/0208001, 2002.

\bibitem{KSSet}
G.~Scarpa and S.~Severini.
\newblock {K}ochen-{S}pecker sets and the rank-1 quantum chromatic number.
\newblock {\em IEEE Transactions on Information Theory}, 58(4):2524--2529, 2012.

\bibitem{Stinson}
D.~R. Stinson.
\newblock {\em Combinatorial Designs: Constructions and Analysis}.
\newblock Springer New York, New York, NY, 1 edition, 2004.

\bibitem{Ziegler}
G.~M. Ziegler.
\newblock {\em Coloring Hamming Graphs, Optimal Binary Codes, and the 0/1-Borsuk Problem in Low Dimensions}, pages 159--171.
\newblock Springer Berlin Heidelberg, Berlin, Heidelberg, 2001.

 

\end{thebibliography}
\end{document}